\documentclass{article}
\usepackage{amsthm}
\usepackage{graphicx,geometry,amsmath,amssymb,stmaryrd,enumerate,xcolor,latexsym}
\usepackage[english]{babel}
\usepackage{url}
\usepackage[normalem]{ulem}
 \usepackage{lineno}
\usepackage{tikz}
\usetikzlibrary{patterns}
\usepackage{pgfplots}
\pgfplotsset{compat=newest}
\usetikzlibrary{calc}
\newcommand{\cdummy}{\cdot}
\newcommand{\bb}[1]{\mathbb{#1}}
\newcommand{\mbf}[1]{\boldsymbol{#1}}

\newcommand{\Th}{\mathcal{T}_h}
\newcommand{\assign}{:=}

\newcommand{\tmop}[1]{\text{#1}}

\newtheorem{lemma}{Lemma}
\newtheorem{theorem}{Theorem}
\newtheorem{remark}{Remark}
\newtheorem{assumption}{Assumption}

\newcommand{\logLogSlopeTriangle}[5]
{
    \pgfplotsextra
    {
        \pgfkeysgetvalue{/pgfplots/xmin}{\xmin}
        \pgfkeysgetvalue{/pgfplots/xmax}{\xmax}
        \pgfkeysgetvalue{/pgfplots/ymin}{\ymin}
        \pgfkeysgetvalue{/pgfplots/ymax}{\ymax}

        \pgfmathsetmacro{\xArel}{#1}
        \pgfmathsetmacro{\yArel}{#3}
        \pgfmathsetmacro{\xBrel}{#1-#2}
        \pgfmathsetmacro{\yBrel}{\yArel}
        \pgfmathsetmacro{\xCrel}{\xArel}

        \pgfmathsetmacro{\lnxB}{\xmin*(1-(#1-#2))+\xmax*(#1-#2)} % in [xmin,xmax].
        \pgfmathsetmacro{\lnxA}{\xmin*(1-#1)+\xmax*#1} % in [xmin,xmax].
        \pgfmathsetmacro{\lnyA}{\ymin*(1-#3)+\ymax*#3} % in [ymin,ymax].
        \pgfmathsetmacro{\lnyC}{\lnyA+#4*(\lnxA-\lnxB)}
        \pgfmathsetmacro{\yCrel}{\lnyC-\ymin)/(\ymax-\ymin)} % THE IMPROVED EXPRESSION WITHOUT 'DIMENSION TOO LARGE' ERROR.
        
        \coordinate (A) at (rel axis cs:\xArel,\yArel);
        \coordinate (B) at (rel axis cs:\xBrel,\yBrel);
        \coordinate (C) at (rel axis cs:\xCrel,\yCrel);

        \draw[#5]   (A)-- node[pos=0.5,anchor=north] {1}
                    (B)-- 
                    (C)-- node[pos=0.5,anchor=west] {#4}
                    cycle;
    }
}
        \title{$\phi$-FEM for the heat equation: optimal convergence on unfitted meshes in space}
		
		\author{Michel Duprez\footnote{MIMESIS team, Inria Nancy - Grand Est, MLMS team, Universit\'e de Strasbourg, 1 place de l'h\^opital, 67000 Strasbourg, France. \texttt{michel.duprez@inria.fr}} 
		and Vanessa Lleras\footnote{IMAG, Univ Montpellier, CNRS UMR 5149, 499-554 Rue du Truel, 34090 Montpellier, France. \texttt{vanessa.lleras@umontpellier.fr}}
		and Alexei Lozinski\footnote{Universit\'e de Franche-Comt\'e, Laboratoire de math\'{e}matiques de Besan\c{c}on, UMR~CNRS~6623, 16 route de Gray, 25030 Besan\c{c}on Cedex, France. \texttt{alexei.lozinski@univ-fcomte.fr}}
		and Killian Vuillemot\footnote{MIMESIS team, Inria Nancy - Grand Est, MLMS team, Universit\'e de Strasbourg, 1 place de l'h\^opital, 67000 Strasbourg, France. IMAG, Univ Montpellier, CNRS UMR 5149, 499-554 Rue du Truel, 34090 Montpellier, France. \texttt{killian.vuillemot@umontpellier.fr}}}

\date{\today}

\begin{document}%\linenumbers
%%-----------------------------
%%      the top matter
%%-----------------------------
\maketitle

	\begin{abstract}
			Thanks to a finite element method, we solve numerically parabolic partial differential equations on complex domains by avoiding the mesh generation, using a regular background mesh, not fitting the domain and its real boundary exactly. Our technique follows the $\phi$-FEM paradigm, which supposes that the domain is given by a level-set function. In this paper, we prove \textit{a priori} error estimates in $l^2(H^1)$ and $l^\infty(L^2)$ norms for an implicit Euler discretization in time. We give numerical illustrations to highlight the performances of $\phi$-FEM, which combines optimal convergence accuracy, easy implementation process and fastness. 
		\end{abstract}

	\section{Introduction}
	The classical finite element method  for elliptic and parabolic problems (see e.g. \cite{Thomee}) needs a computational mesh fitting the boundary of the physical domain. In some applications in engineering or bio-mechanics, the construction of such meshes may be very time-consuming or even impossible. 
	Alternative approaches, such as Fictitious Domain \cite{glowinski} or Immersed Boundary Methods (IBM)	(see e.g. \cite{IBMrev} for a review), can work on unfitted meshes but are usually not very precise. More recent variants, such as  CutFEM \cite{NitscheXFEMinterface}, demonstrate optimal convergence orders but are less straightforward to implement than the original IBM. In particular, CutFEM needs special quadrature rules on the cells cut by the boundary. Finally, we can also mention the Shifted Boundary Method \cite{sbm} that avoids the non-trivial integration by introducing a boundary correction based on a Taylor expansion.%

A new Finite Element Method on unfitted meshes, named $\phi$-FEM, combining the optimal convergence and the ease of implementation,  was recently proposed in \cite{phifem,phifem2}. Initially developed for stationary elliptic PDEs, it has been extended in \cite{cotin:hal-03372733} to a broader class of equations, including the time-dependent parabolic problems, without any theoretical analysis. The goal of the present note is to provide such an analysis in the case of the Heat-Dirichlet problem
\begin{equation}
\label{eq:heat} 
\partial_tu - \Delta u = f  \text{ in } \Omega \times (0, T),~~
u = 0  \text{ on } \Gamma \times (0, T), ~~
u_{|t=0} = u^0  \text{ in } \Omega, 
\end{equation}
where $T>0$, $\Omega\subset \bb{R}^d$, $d = 2,3$  is a bounded domain  with a smooth boundary $\Gamma$  given by a level-set function on $\bb{R}^d$
\begin{equation}\label{eq:phi} \Omega \assign \{\phi < 0\}  \hspace{0.5cm} \text{ and } \hspace{0.5cm} \Gamma \assign \{\phi = 0\}
\hspace{0.17em} . \end{equation} 
(Note that some FEM on unfitted meshes have been developed for such problems for example, in \cite{barad,Colella}).

For the discretization in time, we use the implicit Euler scheme. The Dirichlet boundary conditions are imposed via a product with the level-set function $\phi$. An appropriate stabilization is introduced to the finite element discretization to obtain well-posed problems. A somewhat unexpected feature of this stabilization is that it works under the constraint on the steps in time and space of the type $\Delta t \geqslant ch^2$. This does not affect the practical interest of the scheme since it is normally intended to be used in the regime $\Delta t \sim h$. 
We shall provide \textit{a priori} error estimates for this scheme in $l^2(H^1)$ norms of similar orders as for the standard FEM, cf. \cite{Thomee}.
We also study the $l^\infty(L^2)$ convergence and prove a slightly suboptimal theoretical bound for it, while it turns out to be optimal numerically. 

\section{Definitions, assumptions, description of the scheme and the main result.}

We assume that $\Omega$ lies inside a box $\mathcal{O}\subset\bb{R}^d$
and that $\Omega$ and $\Gamma$ are given by \eqref{eq:phi}.
The box $\mathcal{O}$ is covered by a simple quasi-uniform simplicial (typically Cartesian) background mesh denoted by $\Th^{\mathcal{O}}$. We introduce the active computational mesh $ \Th \assign \left\{ T \in \Th^{\mathcal{O}} : T \cap \{\phi_h < 0\} \neq
\emptyset \right\}$ on $\Omega_h = \left( \cup_{T \in
	\Th} T \right)^o$, the subdomain of $\mathcal{O}$ composed of mesh cells intersecting $\Omega$, cf. Fig.~\ref{fig:domains_circles} (right). Here,  $\phi_h$ is a piecewise polynomial  interpolation of $\phi$ in finite element space of degree $l\in\mathbb{N}^*$ on $\Th^{\mathcal{O}}$. We shall also need a submesh $\Th^{\Gamma}$, containing the elements of $\Th$ that
are cut by the approximate boundary $\Gamma_h \assign \{\phi_h = 0\}$:
$ \Th^{\Gamma} = \{T \in \Th : T \cap \Gamma_h \neq \emptyset\}$. 
Finally, we denote by $\mathcal{F}_h^{\Gamma}$ the set of the
internal facets $E$ of mesh  $\Th$ belonging to the cells of the set $\Th^{\Gamma}$,
$ \mathcal{F}_h^{\Gamma} \assign \{E \text{ (internal facet of $\Th$) such
that } \exists \hspace{0.17em} T \hspace{0.17em} \in \Th : T \cap \Gamma_h
\neq \emptyset \hspace{0.17em} \text{and } \hspace{0.17em} E \in \partial
T\}$.

Introduce a uniform partition of $[0, T]$ into time steps $0 = t_0 < t_1
< \ldots < t_N = T$ with $t_n=n\Delta t$. 
The basic idea of $\phi$-FEM is to introduce the new unknown $w=w(x,t)$ and to set $u=\phi w$ so that the Dirichlet condition $u=0$ is automatically satisfied on $\Gamma$ since $\phi$ vanishes there. 
Using an implicit Euler scheme to discretize \eqref{eq:heat} in time and denoting $f^n (\cdummy) = f (\cdummy, t_n)$, we get the following discretization in time: given
$u^n = \phi w^n$ find $u^{n + 1} = \phi w^{n + 1}$ such that
\begin{equation}
\label{eq:time_scheme}\dfrac{\phi w^{n + 1} - \phi w^n}{\Delta t} - \Delta
(\phi w^{n + 1}) = f^{n + 1} \hspace{0.17em} .
\end{equation}
To discretize in space, we introduce the finite element space of degree $k$ on $\Omega_h$,
\[ V_h^{(k)} = \{v_h \in H^1 (\Omega_h) \hspace{0.17em} : \hspace{0.17em} v_h
|_T \in \bb{P}_k (T), \hspace{0.17em} \forall \hspace{0.17em} T \in \Th \}\,,\] for some $k\geqslant 1$.
Supposing that $f$ and $u^0$ are actually well defined on $\Omega_h$ (rather
than on $\Omega$ only), we can finally introduce the $\phi$-FEM scheme for
\eqref{eq:heat} as follows: find $w_h^{n + 1} \in V_h^{(k)}$, $n = 0,
1, \ldots, N - 1$ such that for all $v_h \in V_h^{(k)}$
\begin{multline}\label{eq:scheme}
\int_{\Omega_h} \frac{\phi_h w_h^{n + 1}}{\Delta t} \phi_h v_h +
\int_{\Omega_h} \nabla (\phi_h w_h^{n + 1}) \cdot \nabla (\phi_h v_h) 
- \int_{\partial \Omega_h} \frac{\partial}{\partial n}  (\phi_h w_h^{n + 1}) \phi_h v_h \\
+ \sigma h \sum_{E \in \mathcal{F}_h^{\Gamma}} \int_E \left[ \frac{\partial (\phi_h w_h^{n + 1})}{\partial n}  \right]  \left[ \frac{\partial	(\phi_h v_h)}{\partial n} \right]  
- \sigma h^2  \sum_{K\in\Th^{\Gamma}}\int_{K} \left( \frac{\phi_h w_h^{n +1 }}{\Delta t} - \Delta (\phi_h w_h^{n + 1}) \right) \Delta (\phi_h v_h) \\
= \int_{\Omega_h} \left( \frac{u_h^n}{\Delta t} + f^{n + 1} \right) \phi_h
v_h - \sigma h^2  \sum_{K\in\Th^{\Gamma}}\int_{K} \left( \frac{u_h^n}{\Delta t} +
f^{n + 1} \right) \Delta (\phi_h v_h) 
\end{multline}
with $u_h^n = \phi_h w_h^n$ for $n \geqslant 1$ and $u_{h }^0 \in V_h^{(k)}$ an interpolant of $u^0$. 
Moreover, $\phi_h$ is the piecewise polynomial  interpolation of $\phi$ in $V_h^{(l)}$, with $l \geqslant k$.
This scheme contains two stabilization terms: the ghost stabilization (the sum on the facets in $\mathcal{F}_h^\Gamma$) as in \cite{ghost}, and a least-square stabilization (the terms multiplied by $\sigma h^2$) that reinforces (\ref{eq:time_scheme}) on the cells of $\Th^\Gamma$.

\begin{remark}
Our approach can be easily generalized to non-homogeneous Dirichlet boundary conditions $u = u_D$ on $\Gamma \times (0, T)$. We can pose then $u_h^n = \phi_h w_h^n + I_hu_g(\cdot,t_n)$ where $u_g$ is some lifting of $u_D$ from $\Gamma$ to $\Omega_h$ and $I_h$ stands for a finite element interpolation to $V_h^{(k)}$. Scheme (\ref{eq:scheme}) should then be modified accordingly, 
replacing $\phi_h w_h^{n+1}$ by $\phi_h w_h^{n+1} + I_hu_g(\cdot,t_{n+1})$ which results in some additional terms on the right-hand side.
\end{remark}

We recall from \cite{phifem} the assumptions  on the domain and on the mesh required in the theoretical study of the convergence of the $\phi$-FEM scheme.
These assumptions are satisfied if the boundary $\Gamma$ is regular enough and the mesh $\Th$ is fine enough.

\begin{assumption}\label{assumption1}
The boundary $\Gamma$ can be covered by open sets $\mathcal{O}_i$, $i=1,\dots,I$ on which ones we can introduce local coordinates $\xi_1,\dots,\xi_d$ with $\xi_d =\phi$ and such that, up to order $k+1$, all the partial derivatives $\partial^\alpha \xi_i / \partial x^\alpha$ and $ \partial x^\alpha / \partial^\alpha \xi_i $ are bounded by a constant $C_0 > 0$.
Thus, on $\mathcal{O}$, $\phi$ is of class $C^{k+1}$ and there exists $m>0$ such that on $\mathcal{O} \setminus \cup_{i=1,\dots,I} \mathcal{O}_i$, $|\phi| \geqslant m$.
\end{assumption}

\begin{assumption}\label{assumption2}
The approximate boundary $\Gamma_h=\{\phi_h=0\}$ can be covered by element patches $\{ \Pi_k\}_{r=1,\dots,N_\Pi}$ such that : 
\begin{itemize}
\item Each patch $\Pi_r$ can be written $\Pi_r=\Pi_r^\Gamma\cup T_r$ with $\Pi_r^\Gamma \subset \Th^\Gamma$ and  $T_r \in \Th \setminus \Th^\Gamma$. Moreover $\Pi_r$   contains less than $M$ elements and these elements are connected;
\item $\Th^\Gamma = \cup_{r=1,\dots, N_\Pi} \Pi_r^\Gamma$; 
\item Two patches $\Pi_r$ and $\Pi_s$ are disjoint if $r \neq s$.
\end{itemize}
\end{assumption}

\begin{theorem}
	\label{thm:error}
	Assume $\Omega \subset \Omega_h$, $l\geq k$, 
	Assumption \ref{assumption1}-\ref{assumption2}, $f \in H^1  (0, T ; H^{k- 1}(\Omega_h))$ and $u\in H^2 (0, T; H^{k-1} (\Omega))
	$ being the exact solution	to \eqref{eq:heat},  $u^n (\cdummy) = u (\cdummy, t_n)$   and $w_h^n$ be the solution to	\eqref{eq:scheme} for $n = 1, \ldots, N$. 
	For $\sigma$ large enough, there exist $c, C > 0$  depending only on the regularity of mesh $\Th$ and on the constants of Ass. \ref{assumption1}-\ref{assumption2} (with $C$ also depending on $T$), such that if $\Delta t \geqslant ch^2 $ then
	\begin{multline*}
	\left( \sum_{n = 0}^N \Delta t|u^n - \phi_h w_h^n |_{H^1(\Omega)}^2\right)^{\frac{1}{2}}
	\leqslant C \| u^0 - u^0_h \|_{L^2(\Omega_h)}\\ + C (h^k + \Delta t) 
	\left(\|u\|_{H^2 (0, T; H^{k-1} (\Omega))
	}	+\|f\|_{H^1 (0, T; H^{k - 1}(\Omega_h) )} \right)
	\end{multline*}	
	and 
	\begin{multline*}
	\max_{1 \leqslant n \leqslant N} \|u^n - \phi_h w_h^n \|_{L^2(\Omega)}
	\leqslant C \| u^0 - u^0_h \|_{L^2(\Omega_h)}\\ + C(h^{k + \frac{1}{2}} +
	\Delta t) 
	\left(\|u\|_{H^2 (0, T; H^{k-1} (\Omega))
	}	+\|f\|_{H^1 (0, T; H^{k - 1}(\Omega_h) )} \right)\,.
	\end{multline*}
\end{theorem}
\begin{remark}
~
\begin{itemize}
    \item If $k=1$, the norms on the right hand side of the estimates above can be replaced by the norm of $f$ alone in $H^1(0, T ; L^2(\Omega_h))$. Indeed, recalling $\Omega\subset\Omega_h$, this assumption on $f$ implies $u \in H^2  (0, T ; L^2(\Omega)) \cap H^1  (0, T ; H^{2}(\Omega))$, see e.g. \cite[Theorems 5 and 6, Chapter 7.1]{evans2010partial}. On the other hand, imposing such regularity on $u$ over $\Omega$, would not suffice to control the extension of $f$ outside of $\Omega$, so that the regularity of $f$ on  $\Omega_h$ should be postulated any way. This contrasts  with the usual \textit{a priori} estimates for standard FEM (see e.g. \cite{Thomee}).
    \item If $k>1$, we need to suppose the regularity of both $u$ and $f$ as stated above. 
\end{itemize}
\end{remark}

In the rest of the paper, the letter $C$, eventually with subscripts, will stand for various constants depending on the mesh regularity,  the constants from Ass. \ref{assumption1}-\ref{assumption2}, and also on $T$ (when specifically mentioned). Before the proof of Theorem \ref{thm:error}, we recall some results from
\cite{phifem} about $\phi$-FEM for the Poisson equation
with Dirichlet boundary conditions.

\begin{lemma}[cf. {\cite[Lemma 3.7]{phifem}}]\label{LemCoer} 
	Consider the bilinear form	
	\begin{equation*}
	a_h (u, v) = \int_{\Omega_h} \nabla u \cdot \nabla v - \int_{\partial
		\Omega_h} \frac{\partial u}{\partial n} v + \sigma h \sum_{E \in \mathcal{F}_h^{\Gamma}} \int_E \left[ \frac{\partial u}{\partial n}  \right]  \left[ \frac{\partial v}{\partial n} \right] + \sum_{K \in
		\mathcal{T}^{\Gamma}_h} \sigma h^2  \int_K \Delta u \,\Delta v  .
	\end{equation*}
	Provided $\sigma$ is chosen big enough, there exists an $h$-independent
	constant $\alpha > 0$ such that
	\[ a_h (\phi_h v_h, \phi_h v_h) \geqslant \alpha | \phi_h v_h |^2_{H^1(
		\Omega_h)}, \quad \forall v_h \in V_h^{(k)} . \]
\end{lemma}

\begin{lemma}[cf. {\cite[Theorem 2.3]{phifem}}]\label{ThPoisDir}
	For any $f \in H^{k - 1}
	(\Omega_h)$, let $w_h \in V_h^{(k)}$ be the solution to
	\begin{equation*}
	a_h (\phi_h w_h, \phi_h v_h) = \int_{\Omega_h} f \phi_h v_h - \sigma h^2 
	\sum_{K\in\Th^\Gamma} \int_{K} f \Delta (\phi_h v_h) \hspace{0.17em} 
	\end{equation*}	
	and $u \in H^{k + 1} (\Omega)$ be the solution to
	\[ - \Delta u = f \tmop{ in } \Omega, \quad u = 0 \tmop{ on } \Gamma \]
	extended to $\tilde{u} \in H^{k + 1} (\Omega_h)$ so that $u = \tilde{u}$ on
	$\Omega$ and 
	$
	\| \tilde{u} \|_{H^{k + 1}(\Omega_h)}
	\leqslant C\| u \|_{H^{k + 1}(\Omega)}
		\leqslant C\| f \|_{H^{k - 1}(\Omega_h)}.
	$
	Provided $\sigma$ is chosen big enough, there exists an $h$-independent
	constant $C > 0$ such that
	\[ | \tilde{u} - \phi_h w_h |_{H^1(\Omega_h)} \leqslant Ch^k \| f \|_{H^{k - 1}(
		\Omega_h)} \quad \tmop{and} \quad \| \tilde{u} - \phi_h w_h \|_{L^2(
		\Omega_h)} \leqslant Ch^{k + \frac{1}{2}} \| f \|_{H^{k - 1}(\Omega_h)} .\]
\end{lemma}

\begin{remark}
This result is proven in \cite{phifem} under the more stringent assumption $f\in H^{k}(\Omega_h)$ which was used to assure $\tilde u\in H^{k+2}(\Omega_h)$ and to provide an interpolation error of $\tilde u$ by a product $\phi_h w_h$. However, in \cite[Lemma 6]{duprez:hal-03588715} we have proven a better interpolation estimate $ \| \tilde u - \phi_h I_h w\|_{H^s(\Omega_h)} \leqslant Ch^{k+1 - s}  \|f\|_{H^{k-1}(\Omega_h)}$ ($s = 0, 1$) for $\tilde u=\phi w$ and the Scott-Zhang interpolant $I_h$. Thus, $f\in H^{k-1}(\Omega_h)$ is actually sufficient.
\end{remark}

\begin{lemma}
\label{lemPoincare}
	For all $v_h \in V_h^{(k)}$, there holds
	$$\|\phi_hv_h\|_{L^2(\Omega_h) } \leqslant C_P |\phi_hv_h|_{H^1(\Omega_h) }.$$ 
\end{lemma}

\begin{proof}
Let $\tilde{\Omega}_h=\{\phi_h<0\}$.
By the Poincar\'e inequality, 
	$$\|\phi_hv_h\|_{L^2(\tilde{\Omega}_h) } \leqslant C\mbox{diam}(\tilde{\Omega}_h) |\phi_hv_h|_{H^1(\tilde{\Omega}_h) },$$
	and $\mbox{diam}(\tilde{\Omega}_h) \leqslant \mbox{diam}(\mathcal{O})$.
	Moreover, thanks to \cite[Lemma 3.4]{phifem}, it holds
	$$\|\phi_hv_h\|_{L^2(\Omega_h\backslash \tilde{\Omega}_h) } \leqslant\|\phi_hv_h\|_{L^2(\Omega_h^{\Gamma}) } \leqslant C h |\phi_hv_h|_{H^1(\Omega_h^{\Gamma}) },$$ 
	where $\Omega_h^{\Gamma}$ is the domain occupied by the mesh $\Th^{\Gamma}$.
	We conclude noting $\Omega\subset\tilde{\Omega}_h\cup\Omega_h^{\Gamma}$.
\end{proof}

\begin{proof}[Proof of Theorem \ref{thm:error}]
There exists a function $\tilde{u} \in H^2 (0, T; H^{k - 1} (\Omega_h) )$, 
an extension of $u$ to $\Omega_h$, such that
\begin{equation}\label{eq:extra}
\| \tilde{u} \|_{H^2(0, T; H^{k - 1}(\Omega_h))} \leqslant C \|u\|_{H^2(0, T; H^{k - 1} (\Omega ) )}.
\end{equation}

Let $w_h^n$ be the solution to our scheme, which we rewrite as
\begin{multline}\label{scheme}
\int_{\Omega_h} \phi_h \frac{w_h^{n + 1} - w^n_h}{\Delta t} \phi_h v_h + a_h
(\phi_h w_h^{n + 1}, \phi_h v_h) - \sum_{T \in \mathcal{T}^{\Gamma}_h}
\sigma h^2  \int_T \phi_h \frac{w_h^{n + 1} - w^n_h}{\Delta t} \Delta
(\phi_h v_h)\\
= \int_{\Omega_h} f^{n + 1} \phi_h v_h - \sum_{T \in \mathcal{T}^{\Gamma}_h}
\sigma h^2  \int_T f^{n + 1} \Delta (\phi_h v_h)
\end{multline}
for $n \geqslant 1$ while $\phi_h w_h^0$ should be replaced with $u_h^0$ for
$n = 0$. 

For any time $t \in [0, T]$, introduce $\tilde{w}_h ( \cdot,t) =
\tilde{w}_h \in V_h^{(k)}$, as in Lemma \ref{ThPoisDir}, with $f$ replaced by
$f - \partial_t \tilde{u}$ evaluated at time $t$:
\begin{equation}
\label{tildewh} a_h (\phi_h \tilde{w}_h, \phi_h v_h) = \int_{\Omega_h} (f -
\partial_t \tilde{u}) \phi_h v_h - \sigma h^2  \sum_{K\in\Th^{\Gamma}}\int_{K} (f -
\partial_t \tilde{u}) \Delta (\phi_h v_h).
\end{equation}

Let $\tilde{w}^n_h=\tilde{w}_h(t_n)$ and $e_h^n \assign \phi_h  (w^n_h - \tilde{w}^n_h)$ for $n \geqslant 1$
and $e_h^0 \assign u^0_h - \phi_h \tilde{w}^0_h$. Taking the difference between (\ref{scheme}) and (\ref{tildewh}) at time $t_{n+1}$, we get
\begin{multline*}
\int_{\Omega_h} \frac{e_h^{n + 1} - e_h^n}{\Delta t} \phi_h v_h + a_h 
(e_h^{n + 1}, \phi_h v_h) - \sum_{T \in \mathcal{T}^{\Gamma}_h} \sigma h^2 
\int_T \frac{e_h^{n + 1} - e_h^n}{\Delta t} \Delta (\phi_h v_h)\\
\\
= \int_{\Omega_h} \left( \partial_t \tilde{u}^{n+1} - \phi_h \frac{\tilde{w}^{n +
		1}_h - \tilde{w}^n_h}{\Delta t} \right) \phi_h v_h - \sum_{T \in
	\mathcal{T}^{\Gamma}_h} \sigma h^2  \int_T \left( \partial_t \tilde{u}^{n+1} -
\phi_h \frac{\tilde{w}^{n + 1}_h - \tilde{w}^n_h}{\Delta t} \right) \Delta
(\phi_h v_h).
\end{multline*}
Taking $ v_h=w^{n+1}_h - \tilde{w}^{n+1}_h$, i.e. $\phi_h v_h = e_h^{n + 1}$, applying the equality
\begin{equation*}
\|e^{n + 1}_h \|^2_{L^2(\Omega_h)} - (e^n_h, e^{n + 1}_h)_{L^2(\Omega_h)} =
\frac{\|e_h^{n + 1} \|_{L^2(\Omega_h)}^2 - \|e_h^n \|_{L^2(\Omega_h)}^2 +
	\|e_h^{n + 1} - e_h^n \|_{L^2(\Omega_h)}^2}{2} \hspace{0.17em},
\end{equation*}
and estimating the terms in the RHS by Cauchy-Schwarz and inverse
inequalities $$\| \Delta e_h^{n + 1} \|_{L^2( T)} \leqslant Ch^{-2} \|
e_h^{n + 1} \|_{L^2( T)}$$ we deduce that
\begin{multline}
\label{123} \frac{\|e_h^{n + 1} \|_{L^2(\Omega_h)}^2 - \|e_h^n \|_{L^2(
		\Omega_h)}^2 + \|e_h^{n + 1} - e_h^n \|_{L^2(\Omega_h)}^2}{2 \Delta t} +
\overbrace{a_h  (e_h^{n + 1}, e_h^{n + 1})}^{(I)}
- \overbrace{\sigma h^2  \int_{\Omega_h^{\Gamma}} \frac{e_h^{n + 1} -
		e_h^n}{\Delta t} \Delta e_h^{n + 1}}^{(II)}\\
\leqslant \underbrace{C \left\| \partial_t \tilde{u}^{n+1} - \phi_h
	\frac{\tilde{w}^{n + 1}_h - \tilde{w}^n_h}{\Delta t} \right\|_{L^2(\Omega_h)}
	\| e_h^{n + 1} \|_{L^2(\Omega_h)}}_{(III)}.
\end{multline}
Thanks to the coercivity lemma \ref{LemCoer}, the term $(I)$ can be bounded
from below by $\alpha | e_h^{n + 1} |_{H^1(\Omega_h)}^2$. We now use the Young
inequality (with some $\varepsilon > 0$) and the inverse inequality $\| \Delta
e_h^{n + 1} \|_{L^2(T)} \leqslant C_I h^{-1} | e_h^{n + 1} |_{H^1(T)}$ to
bound the term $(II)$:
\begin{multline}\label{eq:est12}
(I) - (II) \geqslant 
\alpha | e_h^{n + 1} |_{H^1(\Omega_h)}^2 
- \frac{\sigma h^2}{2 \epsilon (\Delta t)^2} \|e_h^{n + 1} - e_h^n \|^2_{L^2(\Omega_h^{\Gamma})} 
- \frac{\epsilon \sigma C_I^2}{2}  | e_h^{n + 1} |^2_{H^1(\Omega_h^{\Gamma})} \\
\geqslant \frac{3}{4} \alpha | e_h^{n + 1} |_{H^1(	\Omega_h)}^2 
- \frac{1}{2 \Delta t } \|e_h^{n + 1} - e_h^n \|^2_{L^2(\Omega_h^{\Gamma})},
\end{multline}
where we have chosen $\epsilon$ so that $\epsilon \sigma C_I^2/2 =
\alpha/4$ and then assumed $\sigma h^2/(\epsilon \Delta t) \leqslant 1$. This will allow us to control the negative term above by the similar positive term in (\ref{123}), and leads to the restriction  $\Delta t \geqslant ch^2$ with  $c=\sigma/\epsilon$.

We turn now to the RHS of (\ref{123}), i.e. term $(III)$. By triangle inequality
\begin{multline}\label{ResTri} 
\left\| \partial_t \tilde{u}^{n + 1} - \phi_h
\frac{\tilde{w}^{n + 1}_h - \tilde{w}^n_h}{\Delta t} \right\|_{L^2(\Omega_h)}
\leqslant \left\| \partial_t \tilde{u}^{n + 1} - \frac{\tilde{u}^{n + 1} -
	\tilde{u}^n}{\Delta t} \right\|_{L^2(\Omega_h)}\\ + \left\| \frac{\tilde{u}^{n
		+ 1} - \tilde{u}^n}{\Delta t} - \phi_h \frac{\tilde{w}^{n + 1}_h -
	\tilde{w}^n_h}{\Delta t} \right\|_{L^2(\Omega_h)}.
\end{multline} 
By Taylor's theorem with integral remainder
\[ \tilde{u}^n(\cdot) = \tilde{u}^{n + 1}(\cdot) - \Delta t \partial_t \tilde{u}^{n + 1}(\cdot)
- \int_{t_n}^{t_{n + 1}} \partial_{tt} \tilde{u} (t,\cdot) (t_n - t)
\tmop{dt} \]
so that
\begin{multline*} \left\| \partial_t \tilde{u}^{n + 1} - \frac{\tilde{u}^{n + 1} -
	\tilde{u}^n}{\Delta t} \right\|_{L^2(\Omega_h)} = \frac{1}{\Delta t} \left\|
\int_{t_n}^{t_{n + 1}} \partial_{tt} \tilde{u} (t,\cdot) (t_n - t)
\tmop{dt} \right\|_{L^2(\Omega_h)} \\\leqslant \sqrt{\Delta t} \|
\partial_{tt} \tilde{u} \|_{L^2(t_n, t_{n + 1}; L^2 (\Omega_h))} .
\end{multline*}
Differentiating $-\Delta u=f -
\partial_t u$ and \eqref{tildewh}
in time, we obtain thanks to Lemma \ref{ThPoisDir},
\begin{equation*}
\| \partial_t (\tilde{u} (t) - \phi_h \tilde{w}_h) (t)
\|_{L^2(\Omega_h)} \leqslant Ch^{k + \frac{1}{2}} \| (\partial_t f -
\partial_{tt} \tilde{u}) (t) \|_{H^{k - 1}(\Omega_h)}.
\end{equation*}
Thus, for the second term in (\ref{ResTri}), we get by the last interpolation estimate:
\begin{align*}
 \left\| \frac{\tilde{u}^{n + 1} - \tilde{u}^n}{\Delta t} - \phi_h \frac{\tilde{w}^{n + 1}_h - \tilde{w}^n_h}{\Delta t} \right\|_{L^2(\Omega_h)}
&= \frac{1}{\Delta t} \left\| \int_{t_n}^{t_{n + 1}} \partial_t (\tilde{u}
(t,\cdot) - \phi_h \tilde{w}_h (t,\cdot)) \tmop{dt} \right\|_{L^2(\Omega_h)} \\ 
&\leqslant \frac{Ch^{k + \frac{1}{2}}}{\sqrt{\Delta t}} \| \partial_t f - \partial_{tt} \tilde{u} \|_{L^2(t_n, t_{n + 1}; H^{k - 1} (\Omega_h))}. \end{align*}
Collecting these estimates and applying the Young inequality with some $\delta
> 0$ and Poincaré inequality from Lemma \ref{lemPoincare}, we get
\begin{multline}\label{eq:est3}
(III) \leqslant \frac{C}{\delta}  \left(\Delta t \| \partial_{tt}
\tilde{u} \|_{L^2(t_n, t_{n + 1}; L^2 (\Omega_h))}^2 + \frac{h^{2 k + 1}}{\Delta t} \|
\partial_t f - \partial_{tt} \tilde{u} \|_{L^2(t_n, t_{n + 1}; H^{k- 1} (\Omega_h))}^2\right) 
		\\+ \frac{\delta C_P^2}{2} | e_h^{n + 1} |^2_{H^1(	\Omega_h)}.
\end{multline}
Substituting \eqref{eq:est12} and \eqref{eq:est3} to (\ref{123}) and taking $\delta$ so that $\delta C_P^2
= \alpha/2$ yields
\begin{multline*} \frac{\|e_h^{n + 1} \|_{L^2(\Omega_h)}^2 - \|e_h^n \|_{L^2(\Omega_h)}^2}{2
	\Delta t} + \frac{\alpha}{2} | e_h^{n + 1} |_{H^1(\Omega_h)}^2 \\ 
\leqslant C \left(\Delta t \| \partial_{tt}
\tilde{u} \|_{L^2(t_n, t_{n + 1}; L^2 (\Omega_h))}^2 + \frac{h^{2 k + 1}}{\Delta t} \|
\partial_t f - \partial_{tt} \tilde{u} \|_{L^2(t_n, t_{n + 1}; H^{k- 1} (\Omega_h))}^2\right) . \end{multline*}
Multiplying this by $2\Delta t$ and summing on $n = 0, \ldots, N - 1$, we get
\begin{multline*}
\|e_h^N \|_{L^2(\Omega_h)}^2 + \alpha \Delta t \sum_{n = 1}^N |e_h^n |_{H^1(	\Omega_h)}^2\\ 
\leqslant \|e_h^0 \|_{L^2(\Omega_h)}^2 
+ C (\Delta t^2 \|
\partial_{tt} \tilde{u} \|_{L^2(0,T; L^2 (\Omega_h))}^2 + h^{2 k +
	1} \| \partial_t f - \partial_{tt} \tilde{u} \|_{L^2(0,T; H^{k - 1}
	(\Omega_h))}^2) .
\end{multline*}
Thus, observing that the sum above can be stopped at any number $n \leqslant N$, we get
\begin{multline*}
\max_{n = 1, \ldots, N} \|e_h^n \|_{L^2(\Omega_h)} + \left( \Delta t
\sum_{n = 1}^N |e_h^n |_{H^1(\Omega_h)}^2 \right)^{\frac{1}{2}} \\
\leqslant C \|e_h^0 \|_{L^2(\Omega_h)} 
+ C \left(\Delta t \|\partial_{tt} \tilde{u} \|_{L^2(0,T; L^2 (\Omega_h))} + h^{k +\frac{1}{2}} \| \partial_t f - \partial_{tt}\tilde{u} \|_{L^2(0,T; H^{k - 1} (\Omega_h))}\right).
\end{multline*} 

Lemma \ref{ThPoisDir} applied to $-\Delta u=f -
\partial_t u$ in $\Omega$ at times $t_n$ gives
\begin{align*}
  &  \max_{n = 0, \ldots, N} \| \tilde{u}^n - \phi_h  \tilde{w}_h^n \|_{L^2(
	\Omega_h)} \leqslant Ch^{k + 1 / 2} \| f - \partial_t \tilde{u} \|_{C ([0,
	T], H^{k - 1} (\Omega_h))},\\
&	\left( \Delta t \sum_{n = 1}^N | \tilde{u}^n - \phi_h  \tilde{w}_h^n
|_{H^1(\Omega_h)}^2 \right)^{\frac{1}{2}} \leqslant C h^k \| f -
\partial_t \tilde{u} \|_{C([0,T], H^{k - 1} (\Omega_h))}.
\end{align*}
In particular,
\begin{multline*}
	\| e_h^0 \|_{L^2(\Omega_h)} \leqslant
		\| u^0-u_h^0 \|_{L^2(\Omega_h)}
		+\| u^0-\phi_h\tilde{w}_h^0 \|_{L^2(\Omega_h)}\\
		\leqslant 	\| u^0-u_h^0 \|_{L^2(\Omega_h)}
	+Ch^{k + 1 / 2}\| f - \partial_t \tilde{u} \|_{C ([0,
	T], H^{k - 1} (\Omega_h))}.
\end{multline*}

Combining this with the regularity of $f$ and $\tilde{u}$, cf. (\ref{eq:extra}), together with the bound $\|\cdot\|_{C([0,T],\cdot)}\leqslant C\|\cdot\|_{H^1(0,T;\cdot)}$ (with $C$ depending on $T$) gives the announced result .
\end{proof}

\section{Numerical experiments}

In this section, we illustrate the performance of our approach on two test cases\footnote{The experiments are executed on a laptop equipped with an Intel Core i7-12700H CPU and 32Gb of memory. Moreover, for the first test case, we use the serial default solver of \textit{FEniCS}. For the second test case, the \textit{GMRES} linear solver is used with \textit{hypre\_amg} as preconditioner.}. We have implemented $\phi$-FEM in \textit{FEniCS} \cite{fenics}, the codes of the simulations are available in the github repository 
\begin{center}\texttt{\url{https://github.com/KVuillemot/PhiFEM_Heat_Equation}}
\end{center}

In our numerical simulations, if the expected convergence is of order
$C_1h^p+C_2\Delta t^m$, we will fix $\Delta t=h^{p/m}$ in such a way we only need to observe if the error is of order $h^p$ numerically.  

\begin{remark}[Norms for the simulations]
To illustrate the convergence of the methods with the simulations, since it is numerically complex to compute the error on the exact domain $\Omega$, we will use the following formula
\begin{equation*}
    \frac{\| u_h - u_{\text{ref}} \|^2_{l^2(0,T,H^1_0(\Omega_{\text{ref}}))}}{\|u_{\text{ref}} \|^2_{l^2(0,T,H^1_0(\Omega_{\text{ref}}))}} \approx  \frac{\sum_{n=0}^N\Delta t \int_{\Omega_{\text{ref}}}|\nabla u_h(.,t_n)-\nabla u_{\text{ref}}(.,t_n)|^2 \mathrm{d}x}{\sum_{n=0}^N\Delta t\int_{\Omega_{\text{ref}}} |\nabla u_{\text{ref}}(.,t_n)|^2 \mathrm{d}x} \,, 
\end{equation*}
and
\begin{equation*}
   \frac{\| u_h - u_{\text{ref}} \|^2_{l^\infty (0,T,L^2(\Omega_{\text{ref}}))}}{\|u_{\text{ref}} \|^2_{l^\infty (0,T,L^2(\Omega_{\text{ref}}))}} \approx  \frac{\max_{n=0,\dots,N}\int_{\Omega_{\text{ref}}}( u_h(.,t_n)-u_{\text{ref}}(.,t_n))^2 \mathrm{d}x}{\max_{n=0,\dots,N}\int_{\Omega_{\text{ref}}} (u_{\text{ref}}(.,t_n))^2 \mathrm{d}x}\,,
\end{equation*}
 where $u_h$ denotes an approximation of the $L^2$-orthogonal projection of the solution on the reference mesh $\Omega_{\text{ref}}$ and $u_{\text{ref}}$ the reference solution.

\end{remark}

\paragraph{First test case : the source term is deduced from a manufactured solution and the FEM solution is compared to this manufactured solution.}

For this case, we will consider a simple smooth domain : the circle centered in $(0, 0)$, with radius $1$ as represented in Fig.~\ref{fig:domains_circles}. The level-set function is given using the equation of the circle, i.e. $ \phi(x,y) = -1 + x^2 + y^2$. Its approximation $\phi_h$ will be the interpolation of $\phi$ with $\mathbb{P}_{k+1}$ finite elements, except for Fig.~\ref{fig:degree_manufactured} (right).  
\begin{figure}
\begin{center}
\includegraphics[width=0.31\textwidth]{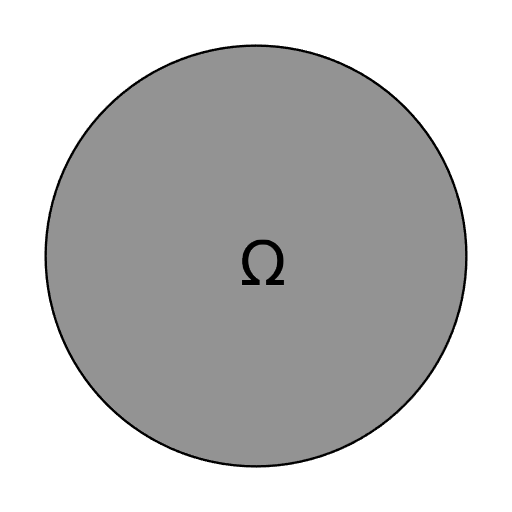}
\quad \includegraphics[width=0.31\textwidth]{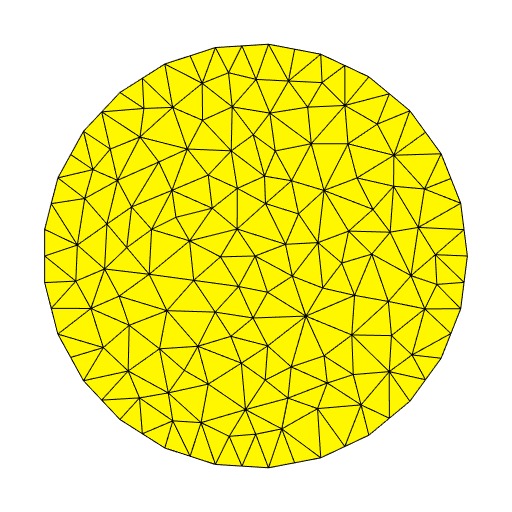}
\quad \includegraphics[width=0.31\textwidth]{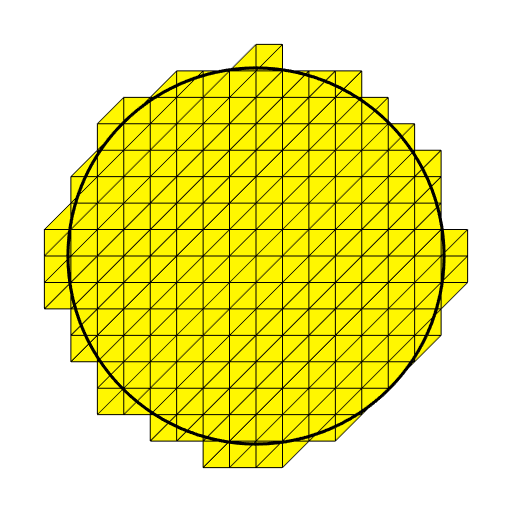}
\end{center}
\caption{Left: considered domain for the first test case. Center: a conforming mesh for the standard FEM. Right: a uniform Cartesian mesh for $\phi$-FEM.}\label{fig:domains_circles}
\end{figure}
Moreover, we consider the manufactured solution given by $u_{\text{ref}} = \cos\left(\frac{1}{2} \pi (x^2+y^2)\right) \exp(x) \sin(t)$ so that $u_{\text{ref}}$ satisfies $u_{\text{ref}}(t=0)=u_{\text{ref}}^0 = 0$ and $u_{\text{ref}} = 0$ on $\Gamma \times (0,T)$. Here, $\Omega_{\text{ref}} = \Omega_h$.  We represent the errors in $l^2(H^1)$ norm on Fig.~\ref{fig:errorH1_manufactured} and in $l^\infty(L^2)$ norm on Fig.~\ref{fig:errorL2_manufactured}, both with $\bb{P}_1$ and $\bb{P}_2$ finite elements ($k=1$ and $k=2$). Here, the numerical results fit well the theoretical convergence order of Theorem \ref{thm:error} and behaves even better since we observe a convergence of orders two and three for the $l^\infty(L^2)$ norm instead of $1.5$ and $2.5$ respectively.  We remark that the theoretical constraint $\Delta t \geqslant c h^2$ is not satisfied for the $\bb{P}^2$ finite elements but it does not affect the practical convergence. We also represent the $l^2(H^1)$ and $l^\infty(L^2)$ errors with respect to the computation time (here, the computation time is the sum of time needed to assemble the finite element matrix and to solve the finite element systems at each time step, without the time used to construct the meshes) in Fig.~\ref{fig:time_manufactured}. We observe that in this case, $\phi$-FEM is significantly faster than a standard FEM to obtain a solution with the same precision.

In Fig.~\ref{fig:sigma_circle_manufactured} (left), we represent the $l^2(H^1)$ error and in Fig.~\ref{fig:sigma_circle_manufactured} (right) the $l^\infty(L^2)$ error, both with respect to $\sigma$. This allows us to emphasize the influence of $\sigma$ on the stability of the errors and validates our choice of $\sigma=1$ in the other simulations.

Finally, in Fig.~\ref{fig:degree_manufactured}, we justify our choice for the degree of interpolation of $\phi$ since in our theoretical result, $\mathbb{P}_{k}$ is sufficient but we observe here that the error decreases for $l=2$. Furthermore, in our previous paper \cite{phifem2}, our theoretical results in the Neumann case hold true only for $l\geq k+1$. Here, since the interpolation is exact from $l=2$ we do not need to compute highest degrees of interpolation for the level-set function to compare the results.  

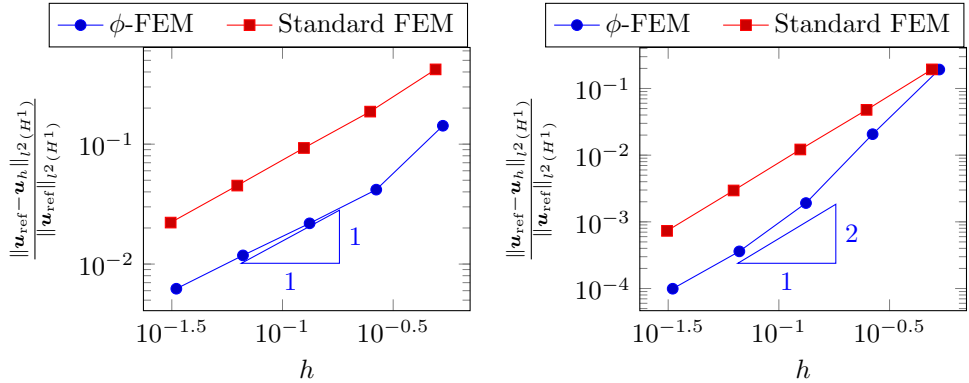
\begin{figure}
\centering
\begin{tikzpicture}\begin{loglogaxis}[
width = .4\textwidth, xlabel = $h$, 
            ylabel = $\frac{\| \mbf{u}_{\text{ref}}-\mbf{u}_h \|_{l^2(H^1)}}{\|\mbf{u}_{\text{ref}}\|_{l^2(H^1)}}$,
            legend style = { at={(-0.29,1)},anchor=south west, legend columns =2,
			/tikz/column 2/.style={column sep = 10pt}}]

\addplot coordinates { %l=k+1
(0.5303300858899106,0.14240683740510918)
(0.2651650429449553,0.04174747780258757)
(0.13258252147247765,0.021866828709415306)
(0.06629126073623882,0.011811524053268598)
(0.03314563036811941,0.006228364239495186)
};
\addplot coordinates{%l=k+1
(0.49156075126796384,0.4203762553276892)
(0.24897191596382212,0.18688693059279907)
(0.12498810741674944,0.09295487649557944)
(0.06247559981507092,0.045148553938495414)
(0.031246983021594454,0.022180926749790872)

};
\logLogSlopeTriangle{0.6}{0.3}{0.18}{1}{blue};
\legend{ $\phi$-FEM, Standard FEM}
\end{loglogaxis}
\end{tikzpicture}%}
\quad
\begin{tikzpicture}\begin{loglogaxis}[width = .4\textwidth, xlabel = $h$, 
            ylabel = $\frac{\| \mbf{u}_{\text{ref}}-\mbf{u}_h \|_{l^2(H^1)}}{\|\mbf{u}_{\text{ref}}\|_{l^2(H^1)}}$,
            legend style = { at={(-0.29,1)},anchor=south west, legend columns =2,
			/tikz/column 2/.style={column sep = 10pt}}]

\addplot coordinates { %l=k+1
(0.5303300858899106,0.19242111130896586)
(0.2651650429449553,0.020637993617911056)
(0.13258252147247765,0.0019127740058687732)
(0.06629126073623882,0.0003606301385867286)
(0.03314563036811941,9.924571425787079e-05)
};
\addplot coordinates{%l=k+1
(0.49156075126796384,0.19342215735167864)
(0.24897191596382212,0.047814836827553664)
(0.12498810741674944,0.012170787553113151)
(0.06247559981507092,0.002954356865058408)
(0.031246983021594454,0.0007320719110801302)
};
\logLogSlopeTriangle{0.6}{0.3}{0.18}{2}{blue};
\legend{$\phi$-FEM, Standard FEM}
\end{loglogaxis}
\end{tikzpicture}
\caption{First test case. $l^2(0,T;H^1(\Omega))$ relative errors with respect to $h$  with $\bb{P}_1$ elements and $\Delta t = h$ (left) and with $\bb{P}_2$ elements and $\Delta t = h^2$ (right). Standard FEM (red squares) and $\phi$-FEM (blue dots), $\sigma = 1$.
}\label{fig:errorH1_manufactured}
\end{figure}

\begin{figure}
\centering
\begin{tikzpicture}
  \begin{loglogaxis}[width = .4\textwidth, xlabel = $h$, 
            ylabel = $\frac{\max_{t_i}\| \mbf{u}_{\text{ref}}(t_i)-\mbf{u}_h(t_i) \|_{0,\Omega}}{\max_{t_i}\|\mbf{u}_{\text{ref}}(t_i)\|_{0,\Omega}}$,
            legend style = { at={(-0.29,1)},anchor=south west, legend columns =2,
			/tikz/column 2/.style={column sep = 10pt}}]

\addplot coordinates {%l=k+1
(0.5303300858899106,0.03404888846350651)
(0.2651650429449553,0.005214516457614594)
(0.13258252147247765,0.0011896385587341234)
(0.06629126073623882,0.0003092842758914203)
(0.03314563036811941,8.064125831383188e-05)
 };
\addplot coordinates { %l=k+1
(0.49156075126796384,0.12904801203833108)
(0.24897191596382212,0.043244479786507195)
(0.12498810741674944,0.01532059254871923)
(0.06247559981507092,0.0056677124975171396)
(0.031246983021594454,0.0018969542517978922)
};
\logLogSlopeTriangle{0.63}{0.3}{0.18}{2}{blue};
\logLogSlopeTriangle{0.63}{0.3}{0.54}{1.5}{red};
\legend{$\phi$-FEM, Standard FEM} 
\end{loglogaxis}\end{tikzpicture}
\quad
\begin{tikzpicture}\begin{loglogaxis}[width = .4\textwidth, xlabel = $h$, 
            ylabel = $\frac{\max_{t_i}\| \mbf{u}_{\text{ref}}(t_i)-\mbf{u}_h(t_i) \|_{0,\Omega}}{\max_{t_i}\|\mbf{u}_{\text{ref}}(t_i)\|_{0,\Omega}}$,
            legend style = { at={(-0.29,1)},anchor=south west, legend columns =2,
			/tikz/column 2/.style={column sep = 10pt}}]

\addplot coordinates {  %l=k+1
(0.5303300858899106,0.01545754327255592)
(0.2651650429449553,0.001173939995100375)
(0.13258252147247765,0.00013135732144535476)
(0.06629126073623882,1.8184856807751004e-05)
};
\addplot coordinates{ %l=k+1
(0.49156075126796384,0.1064849756184299)
(0.24897191596382212,0.027561681435635293)
(0.12498810741674944,0.006712943727538777)
(0.06247559981507092,0.00161909111029182)
};
\logLogSlopeTriangle{0.63}{0.3}{0.18}{3}{blue};
\logLogSlopeTriangle{0.63}{0.3}{0.58}{2}{red};
\legend{$\phi$-FEM, Standard FEM}
\end{loglogaxis}
\end{tikzpicture}

\caption{First test case. $l^{\infty}(0,T;L^2(\Omega))$ relative errors with respect to $h$  with $\bb{P}_1$ elements and $\Delta t = h^2$ (left) and with $\bb{P}_2$ elements and $\Delta t = h^3$ (right). Standard FEM (red squares) and $\phi$-FEM (blue dots), $\sigma = 1$. 
}\label{fig:errorL2_manufactured}
\end{figure}
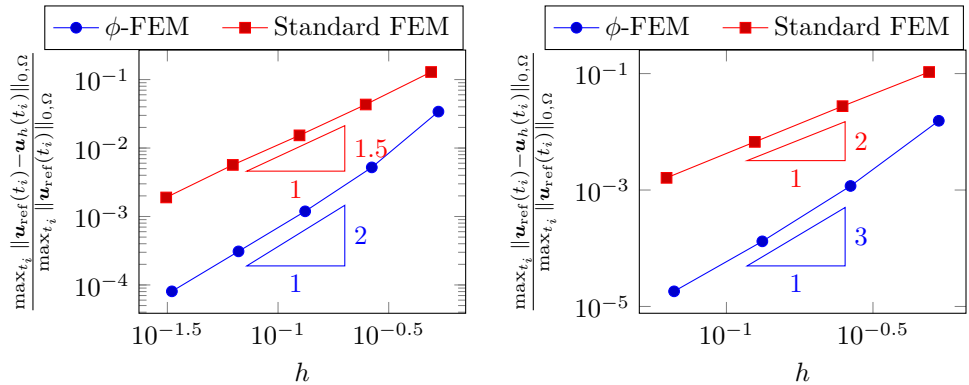

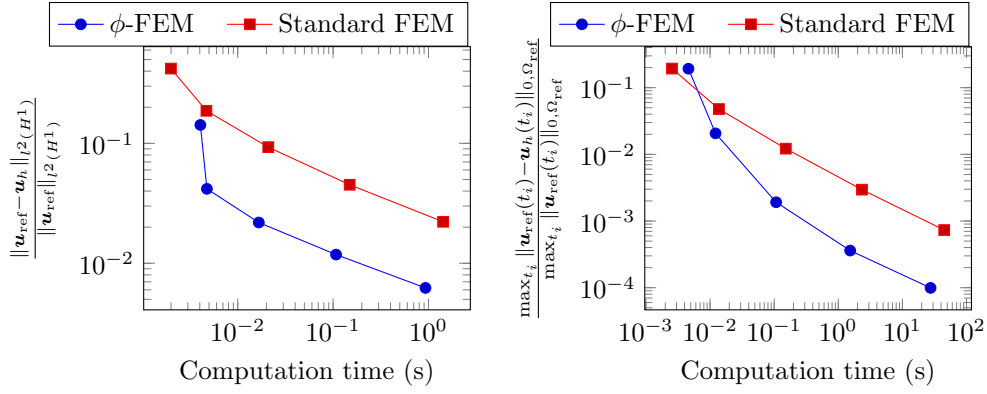
\begin{figure}
\centering
\begin{tikzpicture}\begin{loglogaxis}[width = .4\textwidth, xlabel = Computation time (s), 
            ylabel = $\frac{\| \mbf{u}_{\text{ref}}-\mbf{u}_h \|_{l^2(H^1)}}{\|\mbf{u}_{\text{ref}}\|_{l^2(H^1)}}$,
            legend style = { at={(-0.29,1)},anchor=south west, legend columns =2,
			/tikz/column 2/.style={column sep = 10pt}}]

\addplot coordinates { %l=k+1
(0.004058837890625,0.14240683740510918)
(0.00475311279296875,0.04174747780258757)
(0.01662755012512207,0.021866828709415306)
(0.10751104354858398,0.011811524053268598)
(0.930293083190918,0.006228364239495186)

};
\addplot coordinates{%l=k+1
(0.001992464065551758,0.4203762553276892)
(0.0046977996826171875,0.18688693059279907)
(0.02091526985168457,0.09295487649557944)
(0.1492304801940918,0.045148553938495414)
(1.425015926361084,0.022180926749790872)
};
\legend{$\phi$-FEM, Standard FEM}
\end{loglogaxis}
\end{tikzpicture}
\quad
\begin{tikzpicture}\begin{loglogaxis}[width = .4\textwidth, xlabel = Computation time (s), 
            ylabel = $\frac{\max_{t_i}\| \mbf{u}_{\text{ref}}(t_i)-\mbf{u}_h(t_i) \|_{0,\Omega_{\text{ref}}}}{\max_{t_i}\|\mbf{u}_{\text{ref}}(t_i)\|_{0,\Omega_{\text{ref}}}}$,
            legend style = { at={(-0.29,1)},anchor=south west, legend columns =2,
			/tikz/column 2/.style={column sep = 10pt}}]

\addplot coordinates {  %l=k+1
(0.004607439041137695,0.19242111130896586)
(0.012182235717773438,0.020637993617911056)
(0.10732126235961914,0.0019127740058687732)
(1.5325732231140137,0.0003606301385867286)
(27.49291229248047,9.924571425787079e-05)
};
\addplot coordinates{ %l=k+1
(0.0025506019592285156,0.19342215735167864)
(0.013753414154052734,0.047814836827553664)
(0.15184569358825684,0.012170787553113151)
(2.3403124809265137,0.002954356865058408)
(44.8926739692688,0.0007320719110801302)
};
\legend{$\phi$-FEM, Standard FEM}
\end{loglogaxis}
\end{tikzpicture}

\caption{First test case. $l^2(0,T; H^1(\Omega))$ with $\Delta t =h$ (left) and $l^{\infty}(0,T;L^2(\Omega))$ with $\Delta t =h^2$ (right) relative errors with respect to the computation time. Standard FEM (red squares) and $\phi$-FEM (blue dots), $\bb{P}_1$ elements, $\sigma = 1$. 
}\label{fig:time_manufactured}
\end{figure}

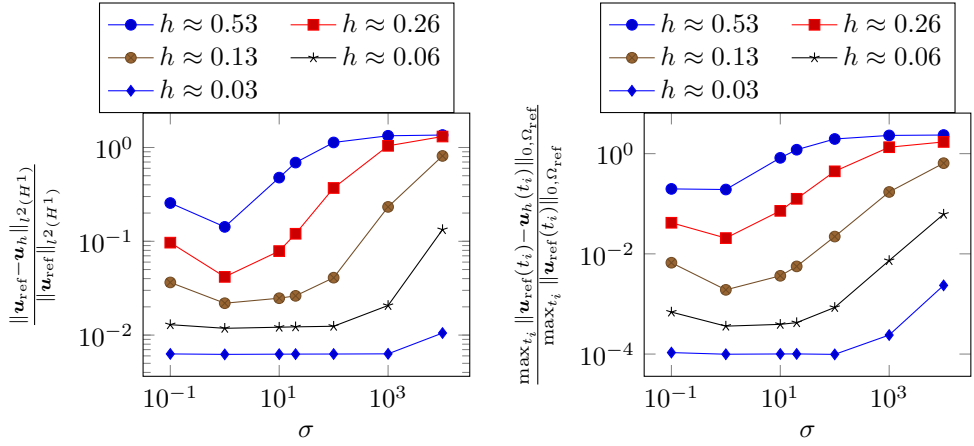
\begin{figure}
\centering
\begin{tikzpicture}
  \begin{loglogaxis}[width = .4\textwidth, xlabel = $\sigma$, 
            ylabel = $\frac{\| \mbf{u}_{\text{ref}}-\mbf{u}_h \|_{l^2(H^1)}}{\|\mbf{u}_{\text{ref}}\|_{l^2(H^1)}}$,
            legend style = { at={(-0.135,1)},anchor=south west, legend columns =2,
			/tikz/column 2/.style={column sep = 10pt}}]

 \addplot coordinates {%l=k+1
(0.1,0.25584897862831985)
(1.0,0.14240683740510918)
(10.0,0.4773756936433008)
(20.0,0.6904478593402711)
(100.0,1.1315683339823686)
(1000.0,1.3333268589466303)
(10000.0,1.3578057949516957)
 };
\addplot coordinates {%l=k+1
(0.1,0.09673181473627455)
(1.0,0.04174747780258757)
(10.0,0.07864129781684082)
(20.0,0.12031374271529559)
(100.0,0.37015559923874514)
(1000.0,1.0419085911716361)
(10000.0,1.3091311876928624)
 };
\addplot coordinates { %l=k+1
(0.1,0.036416939920329246)
(1.0,0.021866828709415306)
(10.0,0.02472390538856593)
(20.0,0.026204400239461453)
(100.0,0.04092891720970238)
(1000.0,0.2326701482491822)
(10000.0,0.8127877100787192)
};
\addplot coordinates { %l=k+1
(0.1,0.012913350844216891)
(1.0,0.011811524053268598)
(10.0,0.012128703840229314)
(20.0,0.012239563654455317)
(100.0,0.012420212194701494)
(1000.0,0.02067819704097055)
(10000.0,0.1331799741282276)
};
\addplot coordinates { %l=k+1
(0.1,0.006297604746265964)
(1.0,0.006228364239495186)
(10.0,0.006260237281607685)
(20.0,0.0062709419844460745)
(100.0,0.00628109834156279)
(1000.0,0.006310904555306457)
(10000.0,0.01051701483733848)
};
\legend{$h \approx 0.53$, $h \approx 0.26$, $h\approx 0.13$, $h\approx 0.06$, $h\approx 0.03$}
\end{loglogaxis}\end{tikzpicture}
\quad
\begin{tikzpicture}
  \begin{loglogaxis}[width = .4\textwidth, xlabel = $\sigma$, 
            ylabel = $\frac{\max_{t_i}\| \mbf{u}_{\text{ref}}(t_i)-\mbf{u}_h(t_i) \|_{0,\Omega_{\text{ref}}}}{\max_{t_i}\|\mbf{u}_{\text{ref}}(t_i)\|_{0,\Omega_{\text{ref}}}}$,
            legend style = { at={(-0.135,1)},anchor=south west, legend columns =2,
			/tikz/column 2/.style={column sep = 10pt}}]

\addplot coordinates {%l=k+1
(0.1,0.19791058243580475)
(1.0,0.19242111130896586)
(10.0,0.8249155425239475)
(20.0,1.1981538554528552)
(100.0,1.9633604484220382)
(1000.0,2.3117267791352214)
(10000.0,2.3539517203430353)
 };
\addplot coordinates { %l=k+1
(0.1,0.04177176852810849)
(1.0,0.020637993617911056)
(10.0,0.07223157274883185)
(20.0,0.12560192216878957)
(100.0,0.4440999255990236)
(1000.0,1.3448412493943018)
(10000.0,1.7141455948622546)
};
\addplot coordinates { %l=k+1
(0.1,0.006653352634329081)
(1.0,0.0019127740058687732)
(10.0,0.003647861117098247)
(20.0,0.005621886432049845)
(100.0,0.022145342177456655)
(1000.0,0.17191400045459196)
(10000.0,0.6456966798951943)
};
\addplot coordinates { %l=k+1
(0.1,0.0006885158815347893)
(1.0,0.0003606301385867286)
(10.0,0.00039157869038409333)
(20.0,0.0004232421939714597)
(100.0,0.0008522778856435858)
(1000.0,0.007350876083667884)
(10000.0,0.06184178192032475)
};
\addplot coordinates { %l=k+1
(0.1,0.00010700550845324871)
(1.0,9.924571425787079e-05)
(10.0,0.00010045453326889023)
(20.0,0.00010062466826920609)
(100.0,9.858185271130041e-05)
(1000.0,0.0002396094912159772)
(10000.0,0.002350748947281737)
};
\legend{$h \approx 0.53$, $h \approx 0.26$, $h\approx 0.13$, $h\approx 0.06$, $h\approx 0.03$}
\end{loglogaxis}\end{tikzpicture}

\caption{First test case. $l^{2}(0,T;H^1(\Omega))$ relative errors with respect to $\sigma$ for different mesh sizes, with $\Delta t =h$ (left) and $l^{\infty}(0,T;L^2(\Omega))$ relative errors with respect to $\sigma$, $\Delta t =h^2$ (right), both with $\bb{P}_1$ elements. 
}\label{fig:sigma_circle_manufactured}
\end{figure}

\begin{figure}
\centering
\begin{tikzpicture}\begin{loglogaxis}[width = .4\textwidth, xlabel = $h$, 
            ylabel = $\frac{\| \mbf{u}-\mbf{u}_h \|_{l^2(H^1)}}{\|\mbf{u}\|_{l^2(H^1)}}$,
            legend style = { at={(0.16,1)},anchor=south west, legend columns =2,
			/tikz/column 2/.style={column sep = 10pt}}]

\addplot coordinates {  %l=k
(0.5303300858899106,0.2947352090285165)
(0.2651650429449553,0.174752129239339)
(0.13258252147247765,0.08717425240243092)
(0.06629126073623882,0.03975755166623604)
(0.03314563036811941,0.016833300367695342)
};
\addplot coordinates{ %l=k+1
(0.5303300858899106,0.14240683740510918)
(0.2651650429449553,0.04174747780258757)
(0.13258252147247765,0.021866828709415306)
(0.06629126073623882,0.011811524053268598)
(0.03314563036811941,0.006228364239495186)
};

\logLogSlopeTriangle{0.73}{0.3}{0.23}{1}{blue};

\legend{$l=1$, $l=2$}
\end{loglogaxis}
\end{tikzpicture}
\quad
\begin{tikzpicture}\begin{loglogaxis}[width = .4\textwidth, xlabel = $h$, 
            ylabel = $\frac{\max_{t_i}\| \mbf{u}_{\text{ref}}(t_i)-\mbf{u}_h(t_i) \|_{0,\Omega_{\text{ref}}}}{\max_{t_i}\|\mbf{u}_{\text{ref}}(t_i)\|_{0,\Omega_{\text{ref}}}}$,
            legend style = { at={(0.16,1)},anchor=south west, legend columns =2,
			/tikz/column 2/.style={column sep = 10pt}}]

\addplot coordinates {  %l=k
(0.5303300858899106,0.30467190128269134)
(0.2651650429449553,0.07832786532302477)
(0.13258252147247765,0.016449943489261428)
(0.06629126073623882,0.0029208251643443425)
(0.03314563036811941,0.0005496259162340615)
};
\addplot coordinates{ %l=k+1
(0.5303300858899106,0.19242111130896586)
(0.2651650429449553,0.020637993617911056)
(0.13258252147247765,0.0019127740058687732)
(0.06629126073623882,0.0003606301385867286)
(0.03314563036811941,9.924571425787079e-05)
};

\logLogSlopeTriangle{0.58}{0.3}{0.13}{2}{blue};

\legend{$l=1$, $l=2$}
\end{loglogaxis}
\end{tikzpicture}

\caption{First test case. $l^{2}(0,T;H^1(\Omega))$ relative errors with respect to $h$ for different values of $l$, $\Delta t =h$ (left) and $l^{\infty}(0,T;L^2(\Omega))$ relative errors with respect to $h$ for different values of $l$, $\Delta t =h^2$ (right), both with $\bb{P}_1$ elements and $\sigma = 1$. 
}\label{fig:degree_manufactured}
\end{figure}
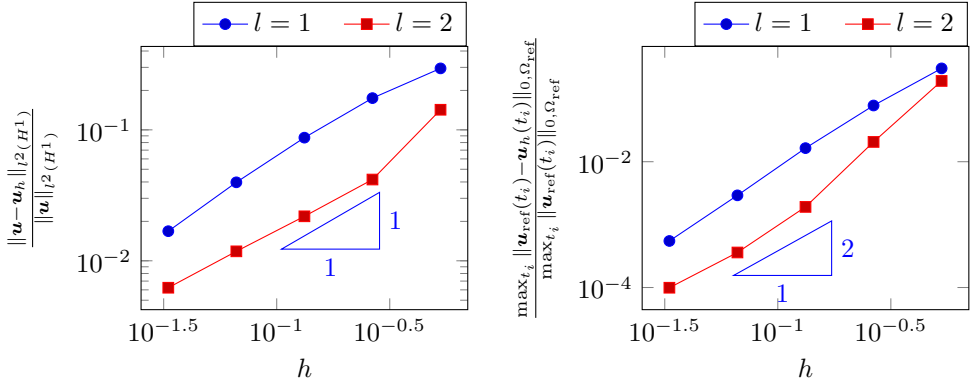

\paragraph*{Second test case : the source term is given and the FEM solution is compared to a standard FEM solution on a very fine mesh.}

We now consider a more realistic test case since we will apply some forces and consider the resulting distribution of heat in the considered domain. More precisely, this time, we impose $u = 0$ on $\Gamma \times (0,T)$, the initial condition is $u^0 = 0$ in $\Omega$ and we define a source term given by $f(x,y,z,t) = \exp\left(-\frac{(x-\mu_1)^2 + (y-\mu_2)^2 + (z-\mu_3)^2}{2\sigma_0^2}\right)$ for each $(x,y,z,t)\in \Omega\times(0,T)$, with $(\mu_1, \mu_2, \mu_3, \sigma_0) = (0.2, 0.3, -0.1, 0.3)$. The final time is fixed to $T=1$.  Moreover, for this test case, we will consider a more complex and 3D domain from \cite{cutfemrev}, given by  
\begin{equation*}
	\phi(x,y,z) = x^2 + y^2 + z^2 - r_0^2 - A \sum_{k=0}^{11} \exp\left(-\frac{(x-x_k)^2 + (y-y_k)^2 + (z-z_k)^2}{\sigma^2_0}\right)\,, 	
\end{equation*}
with 
\begin{align*}
	(x_k,y_k,z_k) &= \frac{r_0}{\sqrt{5}} \left( 2\cos\left(\frac{2k\pi}{5}\right), 2\sin\left(\frac{2k\pi}{5}\right), 1 \right)\,, \quad\hfill 0\leqslant k \leqslant 4\,, \\
	(x_k,y_k,z_k) &= \frac{r_0}{\sqrt{5}} \left( 2\cos\left(\frac{(2(k-5)-1)\pi}{5}\right), 2\sin\left(\frac{(2(k-5)-1)\pi}{5}\right), -1 \right)\,,\quad \hfill 5\leqslant k \leqslant 9\,,  \\ 
	(x_k,y_k,z_k) &= \left( 0, 0, r_0 \right)\,, \quad\hfill k = 10\,, \\
	(x_k,y_k,z_k) &= \left( 0, 0, -r_0 \right)\,, \quad\hfill k = 11\,,
\end{align*}
with $r_0=0.6$, $\sigma = 0.3$ and $A=1.5$.
The resulting domain and meshes are given in Fig.~\ref{fig:situation_second_test_case}.

\begin{figure}
\begin{center}
\includegraphics[width=0.31\textwidth]{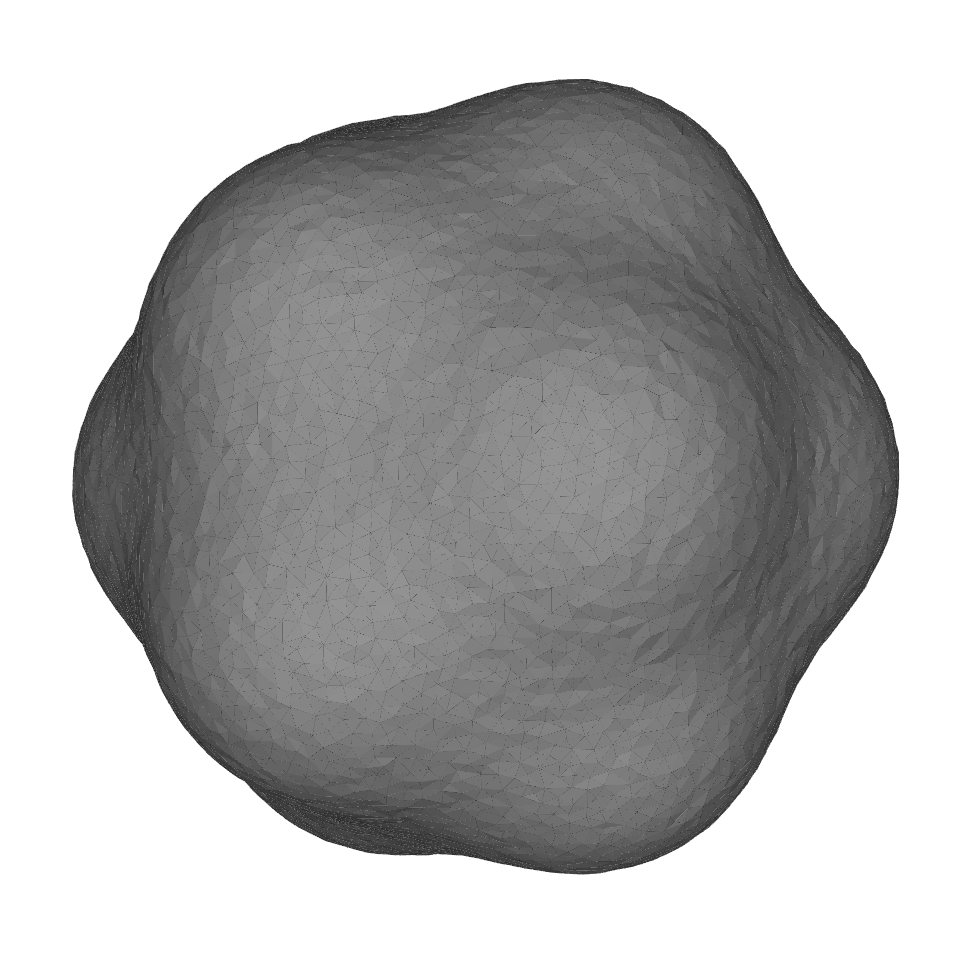}
\quad \includegraphics[width=0.31\textwidth]{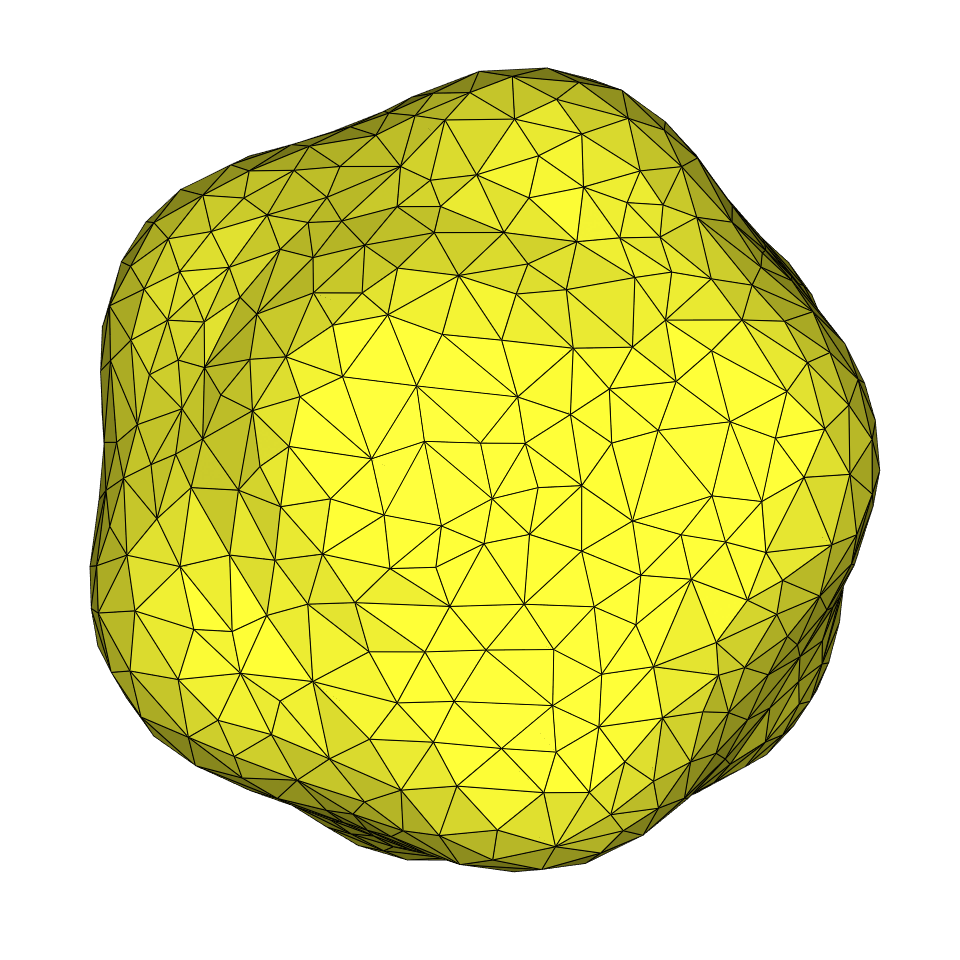}
\quad \includegraphics[width=0.31\textwidth]{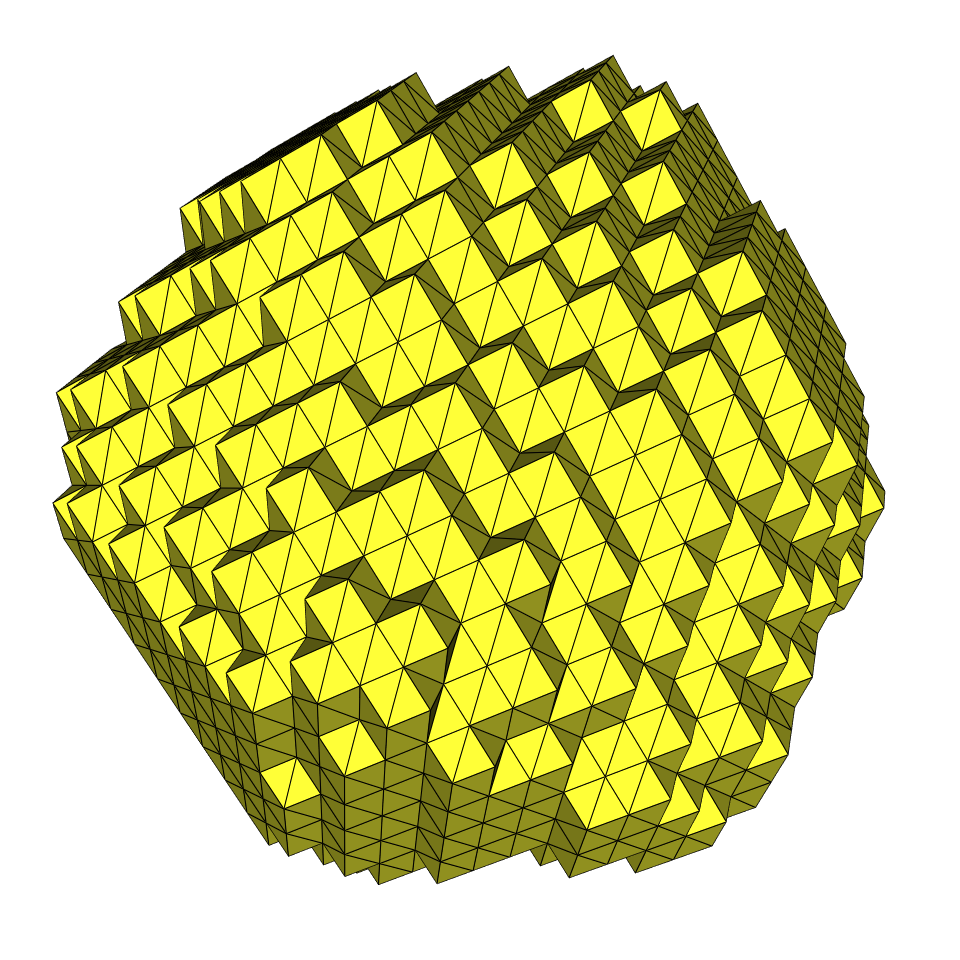}
\end{center}
\caption{Left: considered domain for the second test case. Center: a conforming mesh for the standard FEM. Right: a uniform Cartesian mesh for $\phi$-FEM.}\label{fig:situation_second_test_case}
\end{figure}
Here, $u_{\text{ref}}$ denotes the solution of a classical finite element method on $\Omega_{\text{ref}}$ that is a very fine conforming mesh. In this case, to be more precise, we introduce a partition of the interval $[0,T]$ into time steps $0=t^{\text{ref}}_0 < t^{\text{ref}}_1 < \dots < t^{\text{ref}}_{M} = T $ with $t^{\text{ref}}_n = n \Delta t^{\text{ref}}$ and $\Delta t^{\text{ref}} = h^{p/m}_{\text{ref}}$, where $h_{\text{ref}}$ denotes the size of cells of $\Omega_{\text{ref}}$. Then, in the numerical simulations each discretization is built so that $\left\{t_n\right\}_{n=0,\dots, N}$ is a subset of $\left\{t_n^{\text{ref}}\right\}_{n=0,\dots, M}$.
In Fig.~\ref{fig:error_popcorn}, we consider $\mathbb{P}_1$ finite elements ($k=1$), and $\bb{P}_2$ finite elements for the interpolation $\phi_h$ of $\phi$ ($l=2$). We compare here the $l^2(H^1)$, $l^\infty(L^2)$  relative errors between the solution of the $\phi$-FEM scheme \eqref{eq:scheme} and a standard FEM.  
The numerical results fit well the theoretical convergence order announced in Theorem~\ref{thm:error}, namely, order one for the $l^2(H^1)$ norm and order two for the $l^\infty(L^2)$ error. 

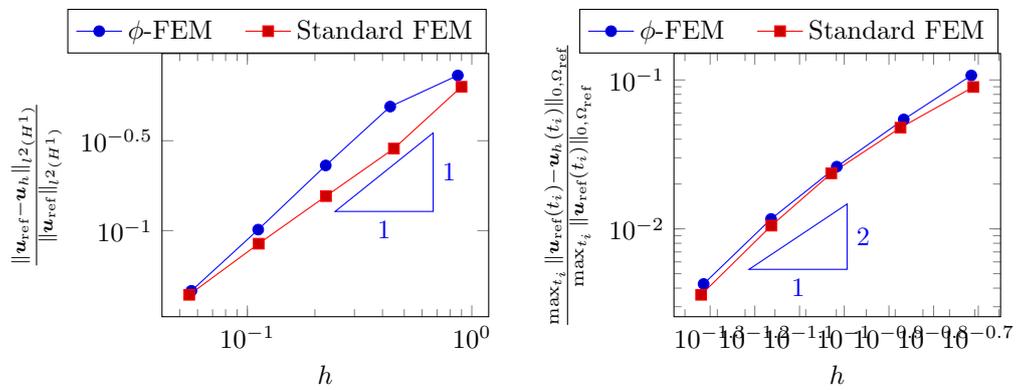
\begin{figure}
\centering
\begin{tikzpicture}\begin{loglogaxis}[
width = .4\textwidth, xlabel = $h$, 
            ylabel = $\frac{\| \mbf{u}_{\text{ref}}-\mbf{u}_h \|_{l^2(H^1)}}{\|\mbf{u}_{\text{ref}}\|_{l^2(H^1)}}$,
            legend style = { at={(-0.29,1)},anchor=south west, legend columns =2,
			/tikz/column 2/.style={column sep = 10pt}}]

\addplot coordinates { %l=k+1
(0.8660254037844386,0.7299580971059874)
(0.4330127018922193,0.4898087562990987)
(0.2234904267830815,0.23042706868283183)
(0.11174521339154113,0.10129118965666067)
(0.056326855530695503,0.04626056393737572)
};
\addplot coordinates{%l=k+1
(0.8994499834446744,0.6313049899925751)
(0.4495668493650076,0.2860136297022238)
(0.22405010848945278,0.15561197493026127)
(0.1122919155747532,0.08457193542182453)
(0.05508233881563368,0.04396265044660608)
};
\logLogSlopeTriangle{0.83}{0.3}{0.4}{1}{blue};
\legend{ $\phi$-FEM, Standard FEM}
\end{loglogaxis}
\end{tikzpicture}
\quad
\begin{tikzpicture}
  \begin{loglogaxis}[width = .4\textwidth, xlabel = $h$, 
            ylabel = $\frac{\max_{t_i}\| \mbf{u}_{\text{ref}}(t_i)-\mbf{u}_h(t_i) \|_{0,\Omega_{\text{ref}}}}{\max_{t_i}\|\mbf{u}_{\text{ref}}(t_i)\|_{0,\Omega_{\text{ref}}}}$, 
            max space between ticks=20, 
            minor xtick = {0.040, 0.044, 0.052, 0.060,0.070, 0.082, 0.096, 0.112, 0.130, 0.150, 0.172, 0.196, 0.222, 0.240},
            legend style = { at={(-0.29,1)},anchor=south west, legend columns =2,
			/tikz/column 2/.style={column sep = 10pt}}]

\addplot coordinates {%l=k+1
(0.19245008972987535,0.10718930024271943)
(0.13584712216226516,0.054315577875441584)
(0.09622504486493806,0.026161006586393188)
(0.06859607158688628,0.011650327607250948)
(0.04844897363829094,0.0042667508806112035)
 };
\addplot coordinates { %l=k+1
(0.19451154363361767,0.08966869479279461)
(0.13359810489587548,0.04779495037349176)
(0.09352417323530174,0.023589718371313743)
(0.06870197777822493,0.010509535315081402)
(0.04783954211973417,0.003608930181694391)
};
\logLogSlopeTriangle{0.53}{0.3}{0.18}{2}{blue};
\legend{$\phi$-FEM, Standard FEM}
\end{loglogaxis}\end{tikzpicture}
\caption{Second test case. $l^2(0,T;H^1(\Omega))$ relative errors with respect to $h$  with $\Delta t = h$ (left) and $l^\infty(0,T;L^2(\Omega))$ relative errors with respect to $h$  with $\Delta t = h^2$ (right), both with $\bb{P}_1$ elements. Standard FEM (red squares) and $\phi$-FEM (blue dots), $\sigma = 1$.
}\label{fig:error_popcorn}
\end{figure}

\section{Conclusion}

In the present work, we proposed a FEM scheme following the $\phi$-FEM paradigm to approximate the solution of the heat equation and proved its convergence, which is optimal in the $l^2(0,T;H^1(\Omega))$ norm and quasi-optimal in the $l^{\infty}(0,T;L^2(\Omega))$ norm. 
We remark that, in comparison with \cite{phifem}, we need less regularity on the exact solution in the a priori error estimates.

A first advantage of the $\phi$-FEM paradigm is its ease of implementation. Indeed, it uses standard shape functions contrary to the XFEM approach. Moreover, it uses standard integration tools contrary to cutFEM needing an integration on the real boundary and some integrations on cut cells. 

A second interesting aspect of our approach is the computational time of the simulation. 
The low cost (computational time) of $\phi$-FEM can be explained
by the fact that the boundary of the geometry in the classical finite element method is approximated by some linear functions 
while, in the $\phi$-FEM paradigm, the boundary is taken into account thanks to the level set function $\phi$, which can be of high degree without increasing the size of the finite element matrix.

In the mathematical analysis, we supposed that the boundary of the considered domain is regular enough. The case of less regular domains will be the aim of future work.

\section*{Funding}
This work was supported by the Agence Nationale de la Recherche, Project PhiFEM, under grant ANR-22-
CE46-0003-01.

\bibliographystyle{unsrt}
\bibliography{biblio}

\end{document}